\documentclass[12pt, a4paper, headsepline]{scrartcl}

\usepackage[english]{babel}
\usepackage[utf8x]{inputenc}

\usepackage{amsmath}
\usepackage{amssymb}
\usepackage{amsthm}
\usepackage{graphicx}

\usepackage{ifthen}  	% used to define switches ``draft'' mode etc.
\usepackage{scrtime} 	 % give time by \thistime

\usepackage{scrpage2}
\usepackage{accents}

\newtheoremstyle{mythmstyle}% name
  {\topsep}% Space above
  {\topsep}% Space below
  {\itshape}% Body font
  {}% Indent amount
  {\bfseries \sffamily}% Theorem head font
  {.}%Punctuation after theorem head
  {.5em}%Space after theorem head
  {}% theorem head spec

\newtheoremstyle{mydefstyle}% name
  {\topsep}% Space above
  {\topsep}% Space below
  {\normalfont}% Body font
  {}% Indent amount
  {\bfseries \sffamily}% Theorem head font
  {.}%Punctuation after theorem head
  {.5em}%Space after theorem head
  {}% theorem head spec

\theoremstyle{mythmstyle}
\newtheorem{thm}{Theorem}[section]      
\newtheorem{prop}[thm]{Proposition}  
\newtheorem{cor}[thm]{Corollary}      
      
\newtheorem{lemma}[thm]{Lemma}           
\theoremstyle{mydefstyle}
\newtheorem{defin}[thm]{Definition}
      
\newtheorem{rmk}[thm]{Remark}

\DeclareMathOperator{\grad}{grad}
\DeclareMathOperator{\dvg}{div}

\newcommand{\e}{\varepsilon}

 \newcommand{\bundle}[2]{\mathit P^{#1}#2}

\newcommand{\grass}[2]{\mathcal{G}_{or}^{#1}(#2)} % symbol for oriented grassmannian
\newcommand{\liealgebra}[1]{\mathfrak{o}} 
\newcommand{\D}[1]{\frac{D}{d#1}}                 % symbol for covariant derivative
\newcommand{\der}[1]{\frac{d}{d#1}}               % symbol for derivative 
\newcommand{\parder}[1]{\frac{\partial}{\partial #1}} %symbol for partial derivative
\newcommand{\tkm}[2]{\mathit T^{#1}#2}            % symbol for T^kM
\newcommand{\fibre}[2]{\mathit O(#1)/O(#1-#2)}

\newcommand{\tvarphi}{\tilde{\varphi}}

\DeclareMathOperator*{\bigunion}{\bigcup\limits}

\newcommand{\Bigvert}{\Big\vert}
\newcommand{\bigvert}{\big\vert}

\newcommand{\R}{\mathbb{R}} % symbol for real numbers
\ifthenelse{\isundefined \draft }
{ 		    %------------------------------------------------------------------------
  \newcommand{\look}[1]{}%   final version: ignore comments for draft
  \newcommand{\lookM}[1]{}%

  \automark[subsection]{section}  %
}
{
  \newcommand{\markerM}{\fbox{\rule{0pt}{0.1ex}\textbf{Michela}}}
  \newcommand{\look}[1]{\markerO \textbf{*}%\rule{3ex}{2.5mm}
    \footnote{ #1 }}
  \newcommand{\lookM}[1]{\markerM\textbf{*}%\rule{3ex}{2.5mm}
    \footnote{\textbf{Michela:} #1 }}

  \rehead{\parbox[t]{1.0\textwidth}%
    {\textsf{\texttt{\jobname.tex} \quad --- DRAFT --- \quad 
        Michela Egidi\\
        \today, \thistime}}}
  \lohead{\parbox[t]{1.0\textwidth}%
    {\textsf{\texttt{\jobname.tex} \quad --- DRAFT ---  \quad 
        Michela Egidi\\
        \today, \thistime}}}
}
\title{Pestov's Identity on frame bundles and applications
	\ifthenelse{\isundefined \draft}{}
 	{\newline \upshape{--- DRAFT---}}}

\author{Michela Egidi}

\dedication{\normalsize{Department of Mathematical Sciences,
    Durham University, England, UK\newline E-mail:
     \texttt{michela.egidi@durham.ac.uk}}}

\ifthenelse{\isundefined \draft}
{\date{\today}}  % final version
{\date{\today, \thistime,  \emph{File:} \texttt{\jobname.tex}}} 
\begin{document}

\maketitle

%----------------------------ABSTRACT-------------------------------------

\begin{abstract}

In this article we lift Pestov's Identity on the tangent bundle of a Riemannian manifold $M$ to the bundle of $k$-tuples of tangent vectors. We also derive an integrated version and a restriction to the frame bundle $\bundle k M$  of $k$-frames. Finally, we discuss a dynamical application for the parallel transport on $\grass k M$, the Grassmannian of oriented $k$-planes of $M$.

\end{abstract}

\tableofcontents

%---------------------------INTRO-----------------------------------------
\section{Introduction}
Pestov's Identity links the generator of the geodesic flow on the tangent bundle of a manifold with the Riemannian curvature tensor and other geometrically motivated differential operators. This identity has many application in dynamics and in the solution of inverse problems such as the X-ray transform and the boundary rigidity problem. 

Pestov and Sharafutdinov \cite{pes-shar:88} introduced this identity in order to derive some useful estimates on symmetric tensor fields, and they give an answer to a question related to tomography, mathematical transport theory and other disciplines, i.e., how uniquely a symmetric tensor field $f$ on a negatively curved manifold $M$ is determined by its integrals over all the geodesics in $M$. Since then, Pestov's Identity has been a key tool in giving answers to this kind of questions. For its main applications in this direction we refer the reader to  \cite{anik-rom:97}, \cite{pat-salo-uhl:13}, \cite{pes-shar:88}, the survey \cite{pat-salo-uhl:14} and the references therein. 

Another application of Pestov's Identity is in the topic of spectral rigidity. Croke and Sharafutdinov \cite{croke-shar:98} used it to prove that a compact manifold of negative curvature is spectrally rigid, i.e., the Laplace-Beltrami spectra of a family of deformed metrics on $M$ are different up to trivial deformation. This also generalizes work of Guillemin and Kazhdan \cite{gui-kaz:80} and Min-Oo \cite{min-oo:85}. 

Pestov's Identity has also been adapted to magnetic and Anosov flows (see for example \cite{ainsworth:13}, \cite{dar-pat:05}) in relation to magnetic tomography and the boundary rigidity problem.

A coordinate-free proof of the identity is given by Knieper in his survey on hyperbolic dynamics and Riemannian geometry \cite{knieper:02}.

In this article we lift the original Pestov Identity on the tangent bundle of a compact manifold $M$ to the bundle of $k$-tuples of tangent vectors and also present an integrated version and its restriction to the bundle $\bundle k M$ of oriented orthonormal $k$-frames. In this new setting, the tangent bundle $TM$ and the unit tangent bundle $SM$ are replaced by $\tkm k M$, the space of $k$-tuples of vectors of $TM$, and $\bundle k M$, the frame bundle, respectively, and the generators of the frame flows play the role of the generator of the geodesic flow. 

The frame flow $F^1_t$ is the parallel transport of a frame $f$ along the geodesic determined by its first vector. It has been studied in relation to ergodicity. Brin and Gromov \cite{brin-gromov:80}, Brin and Karcher \cite{brin-karcher:84} and, more recently, Burns and Pollicot \cite{burns-pollicott:03} proved that the frame flow is ergodic under additional dimension and negative curvature conditions, like pinching. Ergodicity properties of the frame flow on (higher rank) locally symmetric spaces of non-compact type were proved in \cite{maus:05}.

The statement of the Lifted Pestov Identity and its integrated version over $\bundle k M$ can be found in Section 3.2, see Theorems \ref{thm:pestov}, and Section 3.3, see Theorem \ref{thm:integral_pestov}. In particular, we derive a new identity for smooth functions on $\bundle k M$ invariant under one of the frame flows, involving only the $L^2$-norm of the generators of the frame flows and the Riemannian curvature tensor (see Corollary \ref{cor:invariance}).

As an application, we present a dynamical property of smooth functions on $\grass k M$, the oriented $k$-th Grassmannian of $M$. We define $\grass k M$ as the set of all linear $k$-planes of $TM$ together with an intrinsic orientation. We distinguish between intrinsic and non-intrinsic parallel transport of oriented $k$-planes: a parallel transport of an oriented $k$-plane $A_{or}$ along a geodesic $c_v$ on $M$ is \emph{intrinsic} if $v\in A_{or}$ and \emph{non-intrinsic}, otherwise (see Figure \ref{fig:pt}).

\begin{figure}[htb]\label{fig:pt}
\includegraphics[trim = 0mm -16mm 0mm 35mm, clip, scale=0.5]{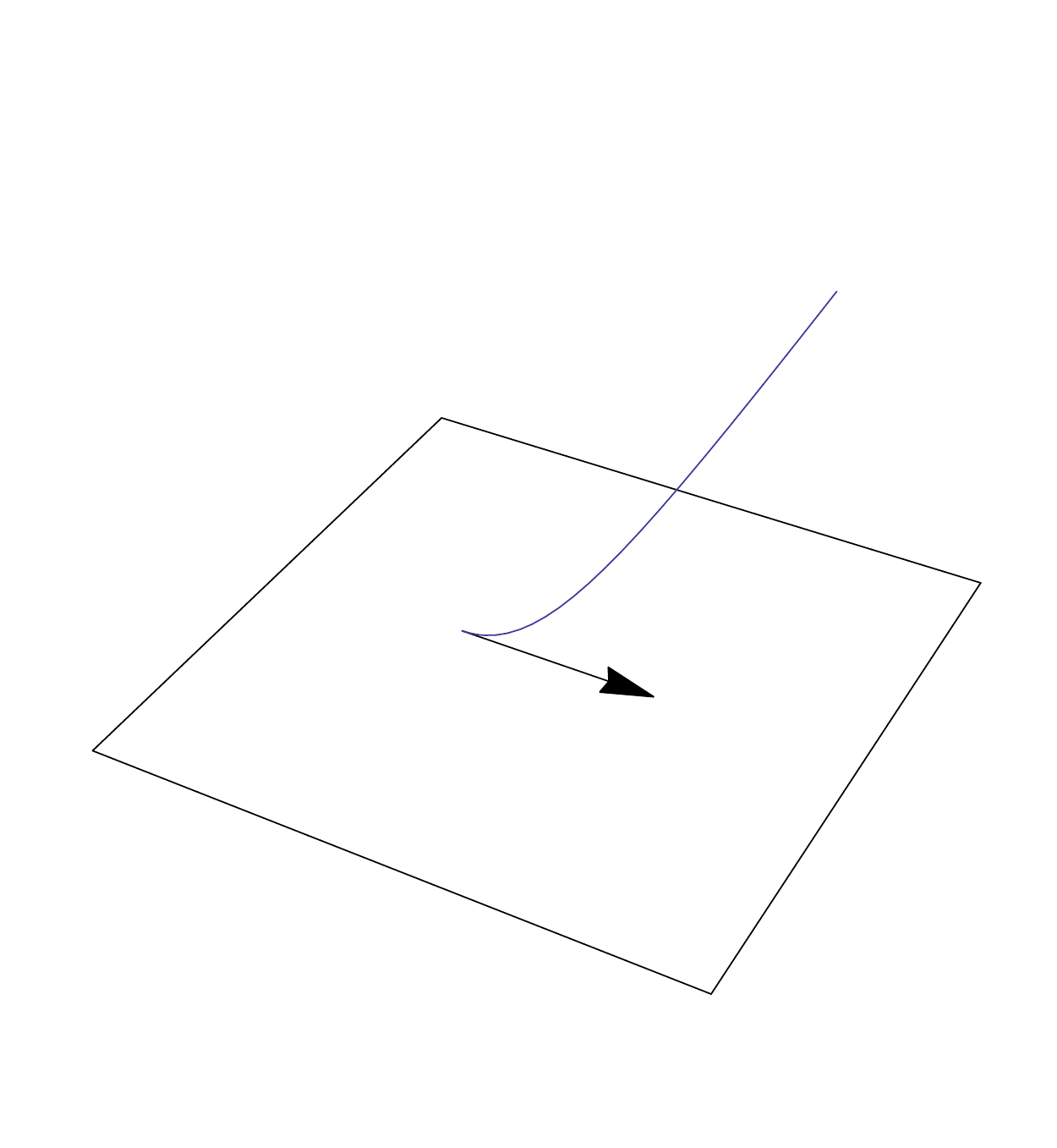}\qquad
\includegraphics[trim = 0mm 2mm 0mm 0mm, clip, scale=0.5]{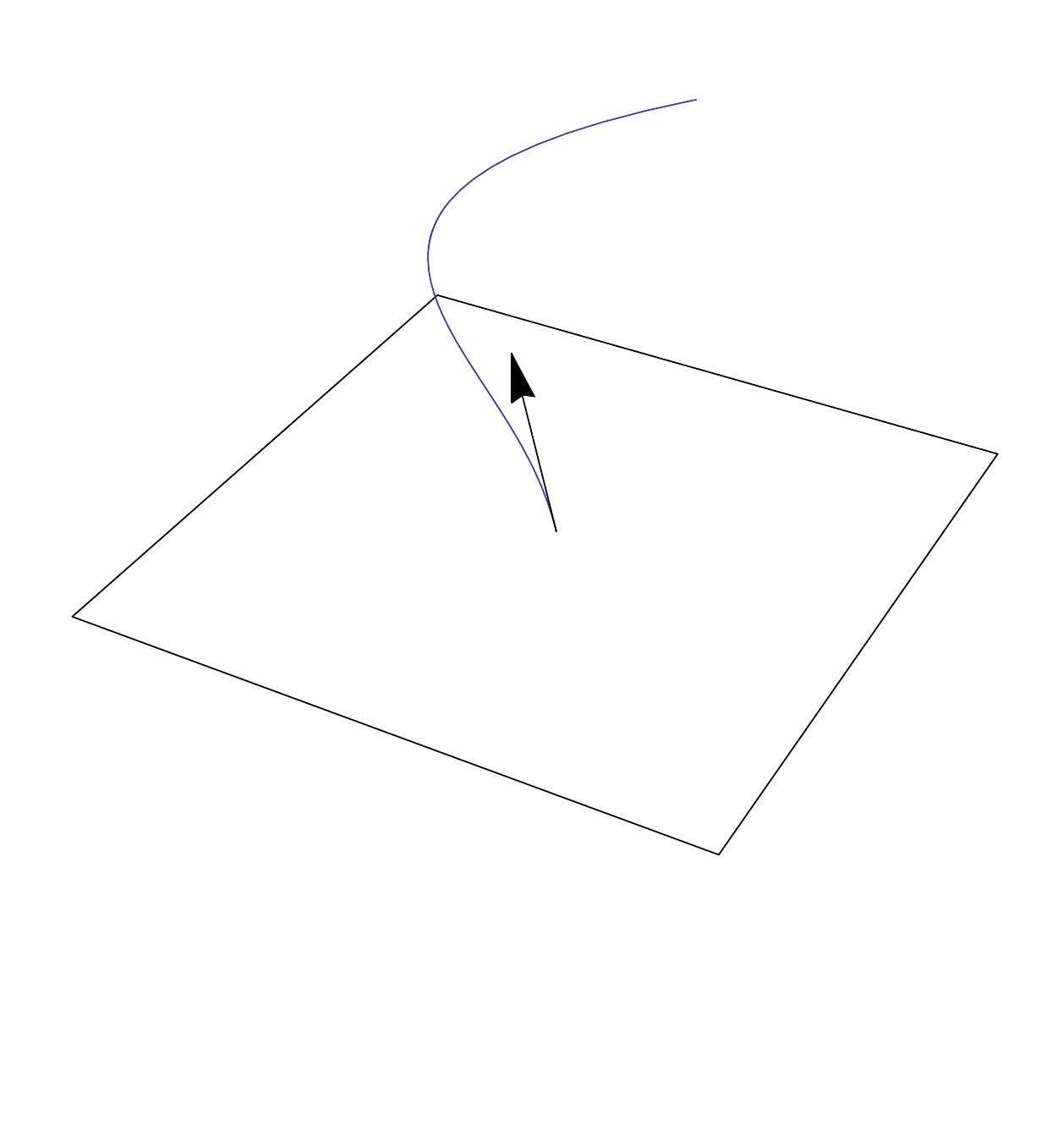}
\vspace{-13mm}\caption{From left to right: example of intrinsic and non-intrinsic parallel transport.}
\end{figure}
 Since we have a canonical projection $\tilde{\pi}:\bundle n M\longrightarrow\grass k M$, there is a close link between the setting of the frame bundle $\bundle n M$ and Grassmannians, connecting the parallel transports on $\grass k M$ with the frame flows in $\bundle n M$.

Applying our results to this setting yields the following.
%----------------------------------------------------------------------------------------------------------------------------------------------------
\begin{thm}\label{thm:invariance_grass_intro}
Let $M$ be a compact $n$-dimensional Riemannian manifold with $\mathcal{R}\leq 0$, $1\leq k\leq n$, and let $\varphi\in C^{\infty}(\grass k M)$. If $\varphi$ is invariant under all intrinsic parallel transports then it is also invariant under all parallel transports.
\end{thm}
%-------------------------------------------------------------------------------------------------------------------------------------------------
Here, $\mathcal{R}$ is the curvature operator of the manifold $M$ (see Section 5.2 for the definition).

Combining Theorem \ref{thm:invariance_grass_intro} with Berger's classification of holonomy (\cite{berger:03} or \cite{besse:87}), we obtain the following proposition.
%-------------------------------------------------------------------------------------------------------------------------------------------------
\begin{prop}\label{prop:constant_functions_intro}
Let $M$ be a non-flat, compact Riemannian manifold with non-positive curvature operator $\mathcal{R}$. Then the following holds:
\begin{itemize}
\item [(i)] If $M$ is either a K\"ahler manifold of real dimension $2n\geq 4$ or a Quaternion-K\"ahler manifold of real dimension $4n\geq 8$ or a locally symmetric space of non-constant curvature (i.e., not the real hyperbolic space), then there exist smooth,\emph{ non-constant} functions on $\grass 2 M$ or $\grass 4 M$ which are invariant under all parallel transports.
\item [(ii)] If $M$ is not one of the exceptions in $(i)$, then, for all $k\leq\dim M$, any smooth function on $\grass k M$ invariant under all parallel transports is necessarily constant.
\end{itemize}
\end{prop}   
%--------------------------------------------------------------------------------------------------------------------------------------------------

The paper is organised as follows. In Section 2 we introduce the space $\tkm k M$ of $k$-tuples of vectors in $TM$ and the frame bundle $\bundle k M$ together with the geometrically motivated horizontal and vertical differential operators. In Section 3 we state and prove the Lifted Pestov Identity and its integrated version for smooth functions on $\tkm k M$. Section 4 is dedicated to the restriction of the Integrated Lifted Pestov Identity to smooth functions on $\bundle k M$. In this section we also state and prove a new identity for smooth functions on the frame bundle $\bundle n M$, invariant under one of the frame flows. In Section 5 we discuss the link between the principal bundle $\tkm n M$ and the oriented $k$-th Grassmannian, $\grass k M$, of $M$ and we prove Theorem \ref{thm:invariance_grass_intro} and Proposition \ref{prop:constant_functions_intro}.

\paragraph{Acknowledgement}
This article is part of the author's PhD thesis. The author would like to thank Norbert Peyerimhoff for many useful discussions and support and Gerhard Knieper for access to an unpublished result of his \cite{knieper:preprint} and helpful comments.  

%-------------------------------------------------------------------------
\section{The spaces $\tkm k M$ and $\bundle k M$, their geometry and differential operators}
Let $(M,g)$ be a compact Riemannian manifold of dimension $n$, let $T_pM$ its tangent space at a point $p\in M$ and let $1\leq k\leq n$. We define the space
%-------------------------------------------------------------------------
\begin{equation*}\label{eq:tkm}
\tkm k M:=\bigunion_{p\in M}\underbrace{T_pM\times\ldots\times T_pM}_{k- times}
\end{equation*} 
%------------------------------------------------------------------------
and a canonical projection map $\pi^k:\tkm k M\longrightarrow M$ such that $f=(u_1,\ldots,u_k)\mapsto p$ if $p=\pi(u_i)$ for all $i=1,\ldots k$ with $\pi:TM\rightarrow M$ and $TM$ the tangent bundle of $M$.
The frame bundle of orthonormal $k$-frames over $M$ is denoted by 
%------------------------------------------------------------------------
\begin{equation*}\label{eq:frame.bundle}
\bundle k M=\{(v_1,\ldots,v_k)\in\tkm k M\;\Bigvert\;\langle v_i, v_j\rangle=\delta_{ij}\}
\end{equation*}
%-------------------------------------------------------------------------
and sits inside $\tkm k M$. The orthogonal group $O(k)$ acts on the right on $\bundle k M$ and there is a canonical projection on $M$, denoted again with $\pi^k$, which is a fibration where the fibre $F_p$ is the Stiefel manifold of orthonormal $k$-frames over $\R^n$, i.e., $F_p\cong\fibre n k$.

In particular, for $k=1$ we have $\tkm 1 M=TM$ and $\bundle 1 M=SM$, where $SM$ denotes the unit tangent bundle of $M$. On the other hand, when $k=n$, $\bundle n M$ is a principal bundle whose fibre isomorphic to $O(n)$.

\subsection{Geometry of $\tkm k M$ and $\bundle k M$}

Let $f=(v_1,\ldots,v_k)\in\tkm k M$, $p=\pi^k(f)$ and let $X=(V_1,\ldots V_k):(-\e,\e)\longrightarrow\tkm k M$ be a curve on $\tkm k M$ with $V_i$ vector fields on $M$ along the footpoint curve $\pi^k\circ X$ on $M$. Then
\begin{equation*}
T_f\tkm k M\ni X'(0)=\Big(\der t\Bigvert_{t=0}\big(\pi^k\circ X\big)(t),\D t\Bigvert_{t=0}V_1(t) ,\ldots,\D t\Bigvert_{t=0}V_k(t)\Big)
\end{equation*}
Therefore, the tangent space of $\tkm k M$ at $f$ with $\pi^k(f)=p$ is given by  
%----------------------------------------------------------------------------
\begin{equation*}\label{eq:tg.space1}
T_f\tkm k M=\underbrace{T_pM\times\ldots\times T_pM}_{(k+1)-times}
\end{equation*}
%----------------------------------------------------------------------------

Let now $f=(v_1,\ldots,v_k)\in\bundle k M$. Any curve $X$ in $\bundle  k M$ through $f$ is given as above with the additional conditions that the $V_i$'s are orthonormal vector fields along the same footpoint curve with $V_i(0)=v_i$. The tangent vector $X'(0)\in T_f\bundle k M$ is again described as above. However, since $\langle V_i(t),V_j(t)\rangle=\delta_{ij}$ for all $t$, differentiation at $t=0$ yields  
$$\langle\D t\Bigvert_{t=0}V_i(t),v_j\rangle=-\langle\D t\Bigvert_{t=0}V_j(t),v_i\rangle$$
Therefore, the tangent space of $\bundle k M$ at $f$ is
%-----------------------------------------------------------------------------
\begin{equation}\label{eq:tg.space2}
T_f\bundle k M=\Big\{(u,w_1,\ldots,w_k)\in T_pM\times\ldots\times T_pM\;\Bigvert\;\big(\langle w_i,v_j\rangle\big)_{ij}\in\mathfrak{o}(k)\Big\}
\end{equation}
%----------------------------------------------------------------------------
where $\mathfrak{o}(n)$ is the Lie algebra of $O(k)$. $T_f\bundle k M$ splits orthogonally into a horizontal and a vertical distribution, $\mathcal{H}_f$ and $\mathcal{V}_f$, described as follows. 
$$\mathcal{H}_f=\Big\{(u,0,\ldots,0)\in T_pM\times\ldots\times T_pM\Big\}\cong T_pM$$
$$\mathcal{V}_f=\Big\{(0,w_1,\ldots,w_k)\in T_pM\times\ldots\times T_pM\;\Bigvert\;\big(\langle w_i,v_j\rangle\big)_{ij}\in\mathfrak{o}(k)\Big\}\cong\mathfrak{o}(k)$$
%Although $T_f\tkm k M$ does not have the same orthogonal split, 
Analogously, any vector $x\in T_f\tkm k M$ can be written as a sum of $(u,0,\ldots,0)$ and $(0,w_1,\ldots,w_k)$, where now there are no conditions on the $w_i$'s. We call these two summands the horizontal and vertical component, respectively, and denote them by $x^h,x^v=(x_1^v,\ldots,x_k^v)$.

Let $x=(x^h,x_1^v,\ldots,x_k^v),y=(y^h,y_1^v,\ldots,y_k^v)\in T_f\tkm k M$ with $\pi^k(f)=p$. We define the metric on $\tkm k M$ as
%-------------------------------------------------------------------------- 
\begin{equation*}\label{eq:metric}
\langle x,y\rangle_{T_f\tkm k M}:=\langle x^h,y^h\rangle_{T_pM}+\sum_{i=1}^k\langle x^v_i,y^v_i\rangle_{T_pM}
\end{equation*}
%--------------------------------------------------------------------------
As a submanifold of $\tkm k M$, $\bundle k M$ inherits this metric and horizontal and vertical component of a vector in $\bundle k M$ are then pairwise orthogonal.
%This metric is naturally inherited by $\bundle k M$ where we can truly talk about horizontal and vertical component as vectors that belong to the horizontal and vertical distributions respectively.
\bigskip

Let $f=(v_1,\ldots,v_k)\in T_f\tkm k M$, and let $c_{v_i}$ be the geodesic on $M$ with $c_{v_i}(0)=\pi^k(f)$ and $c'_{v_i}(0)=v_i$. The $i$-th frame flow, $i=1,\ldots,k$, is the map
\begin{gather*}
F^i_t:\tkm k M\longrightarrow\tkm k M\\
f=(v_1,\ldots,v_k)\mapsto f_{v_i}(t)= ((v_1)_{v_i}(t),\ldots,(v_k)_{v_i}(t))
\end{gather*}
where $f_{v_i}(t)$ denotes the parallel transport of the frame $f$ along the geodesic $c_{v_i}$, i.e., every vector $v_j$ of $f$ is parallel transported along $c_{v_i}$. In particular, $(v_i)_{v_i}(t)=\phi^t(v_i)$, where $\phi^t$ denotes the geodesic flow on $TM$.

Its infinitesimal generator is given by
$$G^i(f)=\der t\Bigvert_{t=0}F^i_t(f)\cong\Big(\der t\Bigvert_{t=0}c_{v_i}(t),0,\ldots,0\Big)=(v_i,0,\ldots,0)$$ i.e., $G^i(f)$ is a horizontal vector of $T_f\tkm k M$.

Finally, we introduce the notion of semi-basic vector field. We define the pullback bundle $\pi^{\ast}(\tkm k M)=\{(v,f)\in TM\times \tkm k M\;\vert\;\pi(v)=\pi^k(f)\}\subset TM\times \tkm k M$. A semi-basic vector field is an element of $\mathfrak{X}(\pi^{\ast}(\tkm k M))=\{X:\tkm k M\rightarrow TM\mbox{ smooth }\;\vert\;X(f)\in T_{\pi^k(f)}M\;\;\forall f\in\tkm k M\}$.

%-------------------------------------------------------------------------
\subsection{Differential operators}
As in the classical case of Riemannian manifolds, we have differential operators on $\tkm k M$ and $\bundle k M$ such as the gradient of a smooth function, the covariant derivative and the divergence. However, here we need to distinguish between  horizontal and vertical distributions when defining these operators.

From now on, all inner products are with respect to the metric on $M$, unless stated otherwise.

Let $\varphi:\tkm k M\longrightarrow\R$ be a smooth function and let $f=(v_1,\ldots,v_k)\in\tkm k M$ with $\pi^k(f)=p$. The gradient of $\varphi$ with respect to the metric on $\tkm k M$ is $$\grad\varphi(f)=(\grad^h\varphi(f),\grad^{v,1}\varphi(f),\ldots,\grad^{v,k}\varphi(f))\in T_f\tkm k M$$
where the components are described intrinsically as follows. 

Let $u\in T_pM$, then 
%---------------------------------------------------------------------------
\begin{gather*}\label{eq:gradients}
\langle\grad^h\varphi(f),u\rangle=\der t\Bigvert_{t=0}\varphi(f_u(t)),\\
\langle\grad^{v,i}\varphi(f),u\rangle=\der t\Bigvert_{t=0}\varphi(v_1,\ldots,v_{i-1},v_i+tu,v_{i+1},\ldots,v_k)\qquad\forall\,i=1,\ldots,k,
\end{gather*}
%-------------------------------------------------------------------------
i.e., the horizontal and $i$-th vertical gradient of $\varphi$ are the derivatives of $\varphi$ along the horizontal curve $t\mapsto f_u(t)$ in $\tkm k M$ and along the vertical curve $t\mapsto v_i+tu$ in the $i$-th $T_pM$ copy of $\tkm k M$, respectively. 

If $f\in\bundle k M$, $\grad\varphi(f)$ defined as above is not an element of $T_f\bundle k M$ because the $k$-tuple $(\grad^{v,1}\varphi(f),\ldots,\grad^{v,k}\varphi(f))$ does not satisfy the constraints in (\ref{eq:tg.space2}). In order to obtain a vector in $T_f\bundle k M$, we need to project the vertical gradient into $T_f\bundle k M$ orthogonally. The orthogonal projection of $\grad^{v,i}\varphi(f)$ into $T_f\bundle k M$ for $f=(v_1,\ldots,v_k)$ is defined as 
%------------------------------------------------------------------------------
\begin{equation}\label{eq:gradv_proj}
\grad_{\liealgebra k}^{v,i}\varphi(f):=\grad^{v,i}\varphi(f)-\frac{1}{2}\sum_{j=1}^k\Big(\langle\grad^{v,i}\varphi(f),v_j\rangle+
\langle\grad^{v,j}\varphi(f),v_i\rangle\Big)v_j
\end{equation}
%------------------------------------------------------------------------------------------
It is easy to verify that the matrix $\big(\langle\grad_{\liealgebra k}^{v,i}\varphi(f),v_j\rangle\big)_{i,j}$ is skew-symmetric and so the vector $(\grad^h\varphi(f),\grad_{\liealgebra k}^{v,1}\varphi(f),\ldots,\grad_{\liealgebra k}^{v,k}\varphi(f))$ belongs to $T_f\bundle k M$.

\bigskip

Let $X:\tkm k M\longrightarrow TM$ be a semi-basic vector field. The horizontal and $i$-th vertical covariant derivative of $X$ with respect to $u\in T_pM$ are given by
%-----------------------------------------------------------------------------
\begin{gather*}\label{eq:cov.der}
\nabla^h_u X(f)=\D t\Bigvert_{t=0}X(f_u(t)),\\
\nabla^{v,i}_u X(f)=\D t\Bigvert_{t=0}X(v_1,\ldots,v_{i-1},v_i+tu,v_{i+1},\ldots,v_k).
\end{gather*}
%-----------------------------------------------------------------------------
Consequently, we define the horizontal and $i$-th vertical divergence as follows.
%--------------------------------------------------------------------------------
\begin{equation*}\label{eq:div}
\dvg^hX(f)=\sum_{i=1}^n\langle\nabla^h_{e_i}X(f),e_i\rangle\qquad\mbox{ and }\qquad
\dvg^{v,i}X(f)=\sum_{i=1}^n\langle\nabla^{v,i}_{e_i}X(f),e_i\rangle,
\end{equation*}
%--------------------------------------------------------------------------------
where $e_1,\ldots, e_n$ is an orthonormal basis of $T_pM$ for $p=\pi^k(f)$.

%-------------------------------------------------------------------------
\section{Lifted Pestov's Identity}
In this section we prove Pestov's Identity for $C^{\infty}$-functions on $\tkm k M$ and we state its integrated version over $\bundle k M$.

The argument for this lifted version is similar to the one given in Knieper's Appendix in \cite{knieper:02} for smooth functions on $TM$.
%---------------------------------------------------------------------------------------------
\subsection{Preliminary lemmas}
%---------------------------------------------------------------------------
\begin{lemma}\label{lem:symmetries} 
Let $\varphi\in C^{\infty}(\tkm k M)$, $f\in\tkm k M$ and $u,w\in T_pM$ with $p=\pi^k(f)$. Let $i=1,\ldots, k$, then 
\begin{equation}\label{eq:sym_grad}
 \langle\nabla^{v,i}_w\grad^h\varphi(f),u\rangle=
\langle\nabla^h_u\grad^{v,i}\varphi(f),w\rangle.
\end{equation}
In particular, it follows
\begin{equation}\label{eq:sym_div}
\dvg^{v,i}\grad^h\varphi(f)=\dvg^h\grad^{v,i}\varphi(f).
\end{equation}
\end{lemma}
%-----------------------------------------------------------------------------------

\begin{proof}
Using the definitions of horizontal and $i$-th vertical covariant derivative and gradient we have  
\begin{align*}
\langle\nabla^{v,i}_w\grad^h\varphi(f),u\rangle & =\der t\Bigvert_{t=0}\langle\grad^h\varphi(v_1,\ldots,v_{i-1},v_i+tw,v_{i+1},\ldots,v_k),u\rangle\\
&=\parder t\Bigvert_{t=0}\parder s\Bigvert_{s=0}\varphi((v_1)_u(s),\ldots,(v_i)_u(s)+t(w)_u(s),\ldots,(v_k)_u(s))\\
&=\der s\Bigvert_{s=0}\langle\grad^{v,i}\varphi(f_u(s)),w\rangle\\
&\mbox{\vspace{3mm}}=\langle\nabla^h_u\grad^{v,i}\varphi(f),w\rangle,
\end{align*}
which proves (\ref{eq:sym_grad}). 

Let $e_1,\ldots,e_n$ be an orthonormal basis of $T_pM$. Then,  
\begin{align*}
\dvg^{v,i}\grad^h\varphi(f)=
\sum_{j=1}^n\langle\nabla^{v,i}_{e_j}\grad^h\varphi(f),e_j\rangle
=\sum_{j=1}^n\langle\nabla^h_{e_j}\grad^{v,i}\varphi(f),e_j\rangle
=\dvg^h\grad^{v,j}\varphi(f),
\end{align*}
which concludes the proof.
\end{proof}

%----------------------------------------------------------------------------------------------
\begin{lemma}\label{lem:tensor}
Let $\varphi\in C^{\infty}(\tkm k M)$, $f=(v_1,\ldots,v_k)\in\tkm k M$ and $u,w\in T_pM$ with $p=\pi^k(f)$. Then
\begin{equation}\label{eq:tensor}
\langle\nabla^h_w\grad^h\varphi(f),u\rangle-
\langle\nabla^h_u\grad^h\varphi(f),w\rangle=\sum_{i=1}^k\langle R(\grad^{v,i}\varphi(f),v_i)w,u\rangle,
\end{equation}
and 
\begin{equation}\label{eq:sym_flow}
G^iG^j\varphi(f)-G^jG^i\varphi(f)=\sum_{l=1}^k\langle R(\grad^{v,l}\varphi(f),v_l)v_i,v_j\rangle.
\end{equation}
\end{lemma}
%-----------------------------------------------------------------------------------------------------

\begin{proof}
We first prove (\ref{eq:tensor}) since (\ref{eq:sym_flow}) follows as a consequence.

Let $H(t,s)=\big(f_w(t)\big)_{u_{w(t)}}(s)$ be a variation in $\tkm k M$, i.e., $H(t,s)=(H_1(t,s),\ldots,H_k(t,s))$ where $H_i(t,s)=\big((v_i)_w(t)\big)_{u_w(t)}(s)$. Then,
\begin{align*}
\langle\nabla^h_w\grad^h\varphi(f),u\rangle&=\der t\Bigvert_{t=0}\langle\grad^h\varphi(f_u(t),u_w(t)\rangle\\
&=\parder t\Bigvert_{t=0}\parder s\Bigvert_{s=0}\varphi\Big((f_w(t))_{u_w(t)}(s)\Big)
=\parder t\Bigvert_{t=0}\parder s\Bigvert_{s=0}\varphi(H(t,s))
\end{align*}
Now,
\begin{equation*}
\begin{split}
\mbox{\hspace{-10mm}}\parder s\Big\vert_{s=0}\parder t\Bigvert_{t=0}&\varphi(H(t,s))
=\parder s\Big\vert_{s=0}\big\langle\grad\varphi(H(0,s)),\parder t\Bigvert_{t=0}H(t,s)\big\rangle_{T_f\tkm k M}\\
&=\parder s\Bigvert_{s=0}\big\langle\grad^h\varphi(H(0,s)),\Big[\parder t\Bigvert_{t=0}H(t,s)\Big]^h\big\rangle\\
&\mbox{\hspace{10mm}}+\parder s\Bigvert_{s=0}\sum_{i=1}^k\big\langle\grad^{v,i}\varphi(H(0,s)),\Big[\parder t\Bigvert_{t=0}H(t,s)\Big]^{v,i}\big\rangle\\
\end{split}
\end{equation*}
\begin{equation*}
\begin{split}
&=\parder s\Bigvert_{s=0}\big\langle\grad^h\varphi(H(0,s)),\der t\Bigvert_{t=0}\big(\pi^k\circ H\big)(t,s)\big\rangle\\
&\mbox{\hspace{10mm}}+\parder s\Bigvert_{s=0}\sum_{i=1}^k\big\langle\grad^{v,i}\varphi(H(0,s)),\D t\Bigvert_{t=0} H_i(t,s)\big\rangle\\
&=\langle\D s\Bigvert_{s=0}\grad^h\varphi(H(0,s)),w\rangle+\langle\grad^h\varphi(f),\underbrace{\D s\Bigvert_{s=0}\parder t\Bigvert_{t=0}\big(\pi^k\circ H\big)(t,s)}_{=\D t\bigvert_{t=0}\parder s\bigvert_{s=0} c_{u_w(t)}(s)
=\D t\bigvert_{t=0}u_w(t)=0}\rangle\\
&+\sum_{i=1}^k\langle\D s\Bigvert_{s=0}\grad^{v,i}\varphi(H(0,s)),\D t\Bigvert_{t=0}H_i(t,0)\rangle+\langle\grad^{v,i}\varphi(f),\D s\Bigvert_{s=0}\D t\Bigvert_{t=0}H_i(t,s)\rangle\\
&=\langle\nabla^h_u\grad^h\varphi(f),w\rangle+\sum_{i=1}^k\langle\grad^{v,i}\varphi(f),\D s\Bigvert_{s=0}\D t\Bigvert_{t=0}H_i(t,s)\rangle,
\end{split}
\end{equation*}
where $c_{u_w(t)}(s)$ is the footpoint curve of $H(t,s)$.

Finally, 
\begin{multline*}
\D s\Bigvert_{s=0}\D t\Bigvert_{t=0}H_i(t,s)=\underbrace{\D t\Bigvert_{t=0}\D s\Bigvert_{s=0}H_i(t,s)}_{=0}\\
+R\Big(\parder s\Bigvert_{s=0}(\pi^k\circ H_i)(0,s),\parder t\Bigvert_{t=0}(\pi^k\circ H_i)(t,0)\Big)H_i(0,0)=R(u,w)v_i.
\end{multline*}
Hence, 
$$\langle\nabla^h_w\grad^h\varphi(f),u\rangle-\langle\nabla^h_u\grad^h\varphi(f),w\rangle=\sum_{i=1}^k\langle R(u,w)v_i,\grad^{v,i}\varphi(f)\rangle.$$
We now prove (\ref{eq:sym_flow}). First, we observe that
%------------------------------------------------------------------------------------------------------------------------------------------------------------------------------------------------------------------------------------------------------------------
\begin{equation}\label{eq:flow_gradh}
G^i\varphi(f)=\frac{d}{dt}\Big\vert_{t=0}\varphi(F^i_t(f))
=\big\langle\grad\varphi(f),\der t\Bigvert_{t=0}F^i_t(f)\big\rangle_{T_f\tkm k M}
=\langle\grad^h\varphi(f),v_i\rangle
\end{equation}
%-------------------------------------------------------------------------------------------------------------------------------------------------------------------------------------------------------------------------------------------------------------------
Therefore,
\begin{align*}
G^iG^j\varphi(f)&=G^i\langle \grad^h\varphi(f),v_j\rangle=
\der t\Bigvert_{t=0}\langle\grad^h\varphi(f_{v_i}(t)),(v_j)_{v_i}(t)\rangle\\
&=\langle\D t\Bigvert_{t=0}\grad^h\varphi(f_{v_i}(t)),v_j\rangle=\langle\nabla^h_{v_i}\grad^h\varphi(f),v_j\rangle
\end{align*}
 and $G^jG^i\varphi(f)=\langle\nabla^{h}_{v_j}\grad^h\varphi(f),v_i\rangle$.

Hence, choosing $u=v_j$ and $w=v_i$ in (\ref{eq:tensor}), we obtain (\ref{eq:sym_flow}).
\end{proof}

%--------------------------------------------------------------------------------
\begin{lemma}\label{lem:G_gradv}%\label{lemma5}
Let $\varphi\in C^{\infty}(\tkm k M)$ and let $f=(v_1,\ldots,v_k)$. Then, for $i,j,l\in\{1,\ldots,k\}$,
\begin{equation}\label{eq:G_gradv}
\langle\grad^{v,i}G^j\varphi(f),v_l\rangle=
G^j\langle\grad^{v,i}\varphi(f),v_l\rangle+\delta_{ij}G^l\varphi(f).
\end{equation}
\end{lemma}
%---------------------------------------------------------------------------------

\begin{proof}
Using (\ref{eq:flow_gradh}) and (\ref{eq:sym_grad}) we have
\begin{equation*}
\begin{split}
\langle\grad^{v,i}G^j\varphi(f),v_l\rangle&=
\der t\Bigvert_{t=0}G^j\varphi(v_1,\ldots,v_i+tv_l,\ldots,v_k)\\
&=\der t\Bigvert_{t=0}\langle\grad^h\varphi(v_1,\ldots,v_i+tv_l,\ldots,v_k),v_j+\delta_{ij}tv_l\rangle\\
&=\langle\nabla^{v,i}_{v_l}\grad^h\varphi(f),v_j\rangle+
\delta_{ij}\langle\grad^h\varphi(f),v_l\rangle\\
&=\langle\nabla^h_{v_j}\grad^{v,i}\varphi(f),v_l\rangle+
\delta_{ij}G^l\varphi(f)\\
&=\der t\Bigvert_{t=0}\langle\grad^{v,i}\varphi(f_{v_j}(t)),(v_l)_{v_j}(t)\rangle+\delta_{ij}G^l\varphi(f)\\
&=G^j\langle\grad^{v,i}\varphi(f),v_l\rangle+\delta_{ij}G^l\varphi(f).\qedhere
\end{split}
\end{equation*}
\end{proof}

%----------------------------------------------------------------------------------------------
\subsection{Lifted Pestov's Identity}
We are now ready to state the main theorem of this section (compare it with \cite[Theorem 1.1, p. 538]{knieper:02}.

%--------------------------------------------------------------------------------------------------------
\begin{thm}[Lifted Pestov's Identity]\label{thm:pestov}
Let $\varphi\in C^{\infty}(\tkm k M)$ and consider the following semi-basic vector fields for $i,j=1,\ldots,k$
\begin{gather*}
Y^{j,i}(f)=\langle\grad^h\varphi(f),\grad^{v,j}\varphi(f)\rangle v_i-\big(G^i\varphi(f)\big)\grad^{v,j}\varphi(f)\\
Z^i(f)=\big(G^i\varphi(f)\big)\grad^h\varphi(f)
\end{gather*}
Then,
\begin{multline}\label{eq:pestov}
\dvg^{v,j}Z^i(f) +\dvg^hY^{j,i}(f) +\delta_{ij}\|\grad^h\varphi(f)\|^2=\\
=\sum_{l=1}^k\langle R(\grad^{v,l}\varphi(f),v_l)v_i,\grad^{v,j}\varphi(f)\rangle+2\langle\grad^h\varphi(f),\grad^{v,j} G^i\varphi(f)\rangle.
\end{multline}
\end{thm}
%---------------------------------------------------------------------------------------------------------

\begin{proof}
Let $e_1,\ldots,e_n$ be an orthonormal basis of $T_pM$ with $p=\pi^k(f)$. We have that $\grad^h\varphi(f)=\sum_{l=1}^n\langle\grad^h\varphi(f),e_l\rangle e_l$. Using this fact, the definition of $j$-th vertical divergence, and general properties of the divergence we have
\begin{align*}
\dvg^{v,j}Z^i(f)&=\sum_{l=1}^n\langle\nabla^{v,j}_{e_l}G^i\varphi(f)\grad^h\varphi(f),e_l\rangle\\
&=\sum_{l=1}^n G^i\varphi(f)\langle\nabla^{v,j}_{e_l}\grad^h\varphi(f),e_l\rangle+
e_l\big(\langle\grad^h\varphi(f),v_i\rangle\big)\langle\grad^h\varphi(f),e_l\rangle\\
&=G^i\varphi(f)\sum_{l=1}^n\langle\nabla^{v,j}_{e_l}\grad^h\varphi(f),e_l\rangle\\
&\mbox{\hspace{7mm}}+\sum_{l=1}^n\langle\grad^h\varphi(f),e_l\rangle\big(\langle
\nabla^{v,j}_{e_l}\grad^h\varphi(f),v_i\rangle+\langle\grad^h\varphi(f),\nabla^{v,j}_{e_l}v_i\rangle\big)
\end{align*}
\begin{align*}
\mbox{\hspace{5mm}}&=G^i\varphi(f)\sum_{l=1}^n\langle\nabla^{v,j}_{e_l}\grad^h\varphi(f),e_l\rangle\\
&\mbox{\hspace{7mm}}+\langle\nabla^{v,j}_{\sum_{l=1}^n\langle\grad^h\varphi(f),e_l\rangle e_l}\grad^h\varphi(f),v_i\rangle+\sum_{l=1}^n\delta_{ij}\langle\grad^h\varphi(f),e_l\rangle^2\\
&=G^i\varphi(f)\sum_{l=1}^n\langle\nabla^{v,j}_{e_l}\grad^h\varphi(f),e_l\rangle+
\langle\nabla^{v,j}_{\grad^h\varphi(f)}\grad^h\varphi(f),v_i\rangle+\delta_{ij}\|\grad^h\varphi(f)\|^2,
\end{align*}
where $v_i$ represents the semi-basic vector field defined as $f=(v_1,\ldots,v_k)\mapsto v_i$.

In the same way,
\begin{align*}
\dvg^hY^{j,i}(f)&=\dvg^h\big(\langle\grad^h\varphi(f),\grad^{v,j}\varphi(f)\rangle v_i\big)-\dvg^h\big(G^i\varphi(f)\grad^{v,j}\varphi(f)\big)\\
&=\langle\grad^h\varphi(f),\grad^{v,j}\varphi(f)\rangle\underbrace{\dvg^hv_i}_{=0}+\langle\grad^h\big(
\langle\grad^h\varphi(f),\grad^{v,j}\varphi(f)\rangle\big),v_i\rangle\\
&\mbox{\hspace{7mm}}-G^i\varphi(f)\dvg^h\grad^{v,j}\varphi(f)-\langle\grad^hG^i\varphi(f),\grad^{v,j}\varphi(f)\rangle\\
&=\frac{d}{dt}\Big\vert_{t=0}\langle\grad^h\varphi(f_{v_i}(t)),\grad^{v,j}\varphi(f_{v_i}(t))\rangle-G^i\varphi(f)\dvg^{v,j}\grad^h\varphi(f)\\
&\mbox{\hspace{7mm}}-\frac{d}{dt}\Big\vert_{t=0}\langle\grad^h\varphi(f_{\grad^{v,j}\varphi(f)}(t),(v_i)_{\grad^{v,j}\varphi(f)}(t)\rangle\\
&=\langle\nabla^h_{v_i}\grad^h\varphi(f),\grad^{v,j}\varphi(f)\rangle+\langle\grad^h\varphi(f),\nabla^h_{v_i}\grad^{v,j}\varphi(f)\rangle\\
&\mbox{\hspace{7mm}}-G^i\varphi(f)\sum_{l=1}^n\langle\nabla^{v,j}_{e_l}\grad^h\varphi(f),e_l\rangle-\langle
\nabla^h_{\grad^{v,j}\varphi(f)}\grad^h\varphi(f),v_i\rangle
\end{align*}
where we use equation (\ref{eq:sym_div}) in the third equality.

Hence, summing these two divergences and using equations (\ref{eq:sym_grad}) and (\ref{eq:tensor}) we have 
\begin{equation*}
\begin{split}
\dvg^{v,j}&Z^i(f)+\dvg^hY^{j,i}(f)=\langle\nabla^h_{v_i}\grad^h\varphi(f),\grad^{v,j}\varphi(f)\rangle+\langle\grad^h\varphi(f),\nabla^h_{v_i}\grad^{v,j}\varphi(f)\rangle\\
&\mbox{\hspace{7mm}}-\langle\nabla^h_{\grad^{v,j}\varphi(f)}\grad^h\varphi(f),v_i\rangle+
\langle\nabla^{v,j}_{\grad^h\varphi(f)}\grad^h\varphi(f),v_i\rangle+\delta_{ij}\|\grad^h\varphi(f)\|^2\\
&=\langle\nabla^h_{v_i}\grad^h\varphi(f),\grad^{v,j}\varphi(f)\rangle-\langle\nabla^h_{\grad^{v,j}\varphi(f)}\grad^h\varphi(f),
v_i\rangle\\
&\mbox{\hspace{7mm}}+2\langle\nabla^h_{v_i}\grad^{v,j}\varphi(f),\grad^h\varphi(f)\rangle+\delta_{ij}\|\grad^h\varphi(f)\|^2\\
&=\delta_{ij}\|\grad^h\varphi(f)\|^2+2\langle\nabla^h_{v_i}\grad^{v,j}\varphi(f),\grad^h\varphi(f)\rangle+\sum_{l=1}^k\langle R(\grad^{v,l}\varphi(f),v_l)v_i,\grad^{v,j}\varphi(f)\rangle.
\end{split}
\end{equation*}
Finally, we have
\begin{align*}
2\langle\grad^h\varphi(f)&,\grad^{v,j}G^i\varphi(f)\rangle=2\langle\grad^h\varphi(f),\grad^{v,j}(\langle\grad^h\varphi(f),v_i\rangle)\rangle\\
&=2\frac{d}{dt}\Big\vert_{t=0}\langle\grad^h\varphi(v_1,\ldots,v_j+t\grad^h\varphi(f),\ldots,v_k),v_i+\delta_{ij}t\grad^h\varphi(f)\rangle\\
&=2\langle\nabla^h_{v_i}\grad^{v,j}\varphi(f),\grad^h\varphi(f)\rangle+2\|\grad^h\varphi(f)\|^2
\end{align*}
Substituting this expression in the previous identity we obtain
\begin{align*}
\dvg^{v,j}&Z^i(f) +\dvg^hY^{j,i}(f)=-\delta_{i,j}\|\grad^h\varphi(f)\|^2\\
 &+\sum_{l=1}^k\langle R(\grad^{v,l}\varphi(f),v_l)v_i,\grad^{v,j}\varphi(f)\rangle+2\langle\grad^h\varphi(f),\grad^{v,j} G^i\varphi(f)\rangle.\qedhere
\end{align*}
\end{proof}

\subsection{Integrated version}

We now want to integrate equation (\ref{eq:pestov}) over the principal bundle $\bundle k M$ with respect to the measure $d\mu_p(f)=dvol(p)\,d\nu_p(f)$ where $d\nu_p$ is the measure on the fibre $F_p$ of $\bundle k M$ and $dvol$ is the measure on $M$. In order to do this, we first need to understand the behaviour of the horizontal divergence and of the generators of the frame flows under integration over the frame bundle.

%-----------------------------------------------------------------------------------
\begin{lemma}\label{lem:integral_divh}
Let $X:\tkm k M\longrightarrow TM$ be a semi-basic vector field, then 
\begin{equation}\label{eq:integral_divh}
\int_{\bundle k M}\dvg^hX(f)d\mu(f)=\int_M\int_{F_p}\dvg^hX(f)\;dvol(p)d\nu_p(f)=0.
\end{equation}
\end{lemma}
%-----------------------------------------------------------------------------------------------

\begin{proof}
We first recall that $F_p\cong\fibre n k$. To prove the Lemma, it is suffices to show 
\begin{equation}\label{eq:1}
\int_{F_p}\langle\nabla^h_v X,v\rangle\;d\nu_p(f)=\langle\nabla_v\int_{F_p}X(f)\;d\nu_p(f),v\rangle.
\end{equation}
In fact, from the above it follows 
$$\int_{\fibre n k}\dvg^hX(f)\;d\nu_p(f)=\dvg\int_{\fibre n k}X(f)\;d\nu_p(f).$$
Then, due to compactness of $M$ we obtain (\ref{eq:integral_divh}).

Let us prove (\ref{eq:1}). We define $\varrho:\bundle n M\rightarrow\bundle k M$ with $\varrho((v_1,\ldots,v_n))=(v_1,\ldots,v_k)$. In particular, $\varrho:\tilde{F}_p\rightarrow F_p$ where $\tilde{F}_p\cong O(n)$ is the fibre of the principal bundle $\bundle n M$ with $O(n)$ right action.

By Theorem 8.1.8 in \cite{feres:98}, there exists a unique (up to scalar) left $O(n)$-invariant Haar measure on $\fibre n k$ that we obtain in the following way.

Fix $f_0\in\tilde{F}_p$ and let $g\in O(n)$. There is a canonical diffeomorphism $$\psi_{f_0}:\fibre n k\rightarrow F_p \mbox{ such that } g\cdot O(n-k)\mapsto\varrho(g\cdot f_0).$$ The pullback $\psi_{f_0}^{\ast}(\nu_p)=\theta_{\fibre n k}$ gives the measure on $\fibre n k$.

In what follows, we will drop the arguments of the measure whenever clear and we will also drop the subscript of the measure $\theta$.

We have
\begin{align*}
\int_{F_p}\langle\nabla^h_v X(f),v\rangle\;d\nu_p(f)=\int_{F_p}\langle\D t\Bigvert_{t=0}X(f_v(t)),v\rangle\;d\nu_p(f)
=\int_{F_p}\der t\Bigvert_{t=0}\langle X(f_v(t)),v\rangle\;d\nu_p(f).
\end{align*}
We have: $X(f_v(t))=X(\varrho(\tilde{f}_v(t)))=X(\varrho(g\cdot(f_0)_v(t)))=X(
\psi_{(f_0)_v(t)}(g))$ for $\tilde{f}=g\cdot f_0\in \tilde{F}_p$. 

Then,
\begin{align*}
\int_{F_p}\der t\Bigvert_{t=0}&\langle X(f_v(t)),v\rangle\;d\nu_p(f)=\int_{\psi_{f_0}\big(\fibre n k\big)}\der t\Bigvert_{t=0}\langle X(f_v(t)),\phi^t(v)\rangle\;d\nu_p(f)\\
&=\int_{\fibre n k}\der t\Bigvert_{t=0}\langle X(\psi_{(f_0)_v(t)}(g)),\phi^t(v)\rangle\;d(\psi^{\ast}_{f_0}\nu_p)(g)\\
&=\der t\Bigvert_{t=0}\int_{\fibre n k}\langle X(\psi_{(f_0)_v(t)}(g)),\phi^t(v)\rangle\;d\theta(g\cdot O(n-k))\\
&=\der t\Bigvert_{t=0}\langle\int_{\fibre n k}X(\psi_{(f_0)_v(t)}(g))\;d\theta\;,\phi^t(v)\;\rangle\\
&=\langle\D t\Bigvert_{t=0}\int_{\fibre n k}X(\psi_{f_0}(g))\;d\theta\;,v\;\rangle\\
&=\langle\D t\Bigvert_{t=0}\int_{F^k_p}X(f)\;d\nu_p(f)\;,v\rangle
\end{align*}
which concludes the proof.
\end{proof}

%------------------------------------------------------------------------------------
\begin{lemma}\label{lem:integral_flow}
Let $\varphi,\psi\in C^{\infty}(\tkm k M)$, then 
\begin{equation}\label{eq:integral_flow}
\int_{\bundle k M}\psi(f) G^i\varphi(f)\;d\mu=-\int_{\bundle k M}\varphi(f)G^i\psi(f)\;d\mu.
\end{equation}
\end{lemma}
%---------------------------------------------------------------------------------------

\begin{proof}
Let $\varphi,\psi\in C^{\infty}(\tkm k M)$. Then
\begin{align*}
&\int_{\bundle k M}\psi(f)G^i\varphi(f)\;d\mu=\int_{\bundle k M}G^i(\psi\varphi)(f)-
\varphi(f)G^i\psi(f)\;d\mu\\
&=\int_{\bundle k M}\langle\grad^h(\psi\varphi)(f),v_i\rangle-\varphi(f)G^i\psi(f)\;d\mu\\
&=\underbrace{\int_{\bundle k M}\dvg^h((\psi\varphi)(f)v_i)\;d\mu}_{=0}-\int_{\bundle k M}(\psi\varphi)(f)\underbrace{\dvg^hv_i}_{=0}\;d\mu-
\int_{\bundle k M}\varphi(f)G^i\psi(f)\;d\mu\\
&=-\int_{\bundle k M}\varphi(f)G^i\psi(f)\;d\mu
\end{align*}
where the first integral vanishes by Lemma \ref{lem:integral_divh} and $v_i$ represents the semi-basic vector field $f=(v_1,\ldots,v_k)\mapsto v_i$.
\end{proof}

We are now ready to state the integrated version of (\ref{eq:pestov}).

%--------------------------------------------------------------------------------------------
\begin{thm}[Integrated Lifted Pestov's Identity]\label{thm:integral_pestov}
Let $\varphi\in C^{\infty}(\tkm k M)$. Then 
\begin{multline}\label{eq:integral_pestov}
 \delta_{ij}\|\grad^h\varphi\|^2_{L^2}-\int_{\bundle k M}\sum_{l=1}^k\langle R(\grad^{v,l}\varphi(f),v_l)v_i,\grad^{v,j}\varphi(f)\rangle\;d\mu=\\
=\int_{\bundle k M}\langle\grad^h\varphi(f),\grad^{v,j}G^i\varphi(f)\rangle+\langle\grad^hG^i\varphi(f),\grad^{v,j}\varphi(f)\rangle\;d\mu.
\end{multline}
where $L^2$ stands for the $L^2$-space on $\bundle k M$.
\end{thm}
%-----------------------------------------------------------------------------------

\begin{proof}
Consider equation (\ref{eq:pestov}). Under integration over $\bundle k M$ the horizontal divergence vanishes and the remaining non-zero terms are
\begin{multline}\label{+++}
\int_{\bundle k M}\dvg^{v,j}Z^i(f)\;d\mu  +\delta_{ij}\|\grad^h\varphi(f)\|^2_{L^2}= 2\int_{\bundle k M}\langle\grad^h\varphi(f),\grad^{v,j} G^i\varphi(f)\rangle\;d\mu\\
+\int_{\bundle k M}\sum_{l=1}^k\langle R(\grad^{v,l}\varphi(f),v_lv_i,\grad^{v,j}\varphi(f)\rangle\;d\mu.
\end{multline}
Using equation (\ref{eq:sym_div}), general properties of the divergence and Lemma \ref{lem:integral_divh}, we have
\begin{align*}
\int_{\bundle k M}&\dvg^{v,j}  Z^i(f)\;d\mu =\int_{\bundle k M}\dvg^{v,j}\big(G^i\varphi(f)\grad^h\varphi(f)\big)\;d\mu\\
& =\int_{\bundle k M}\langle\grad^h\varphi(f),\grad^{v,j}G^i\varphi(f)\rangle+G^i\varphi(f)\dvg^{v,j}\grad^h\varphi(f)\;d\mu\\
& =\int_{\bundle k M}\langle\grad^h\varphi(f),\grad^{v,j}G^i\varphi(f)\rangle+G^i\varphi(f)\dvg^{h}\grad^{v,j}\varphi(f)\;d\mu\\
& =\int_{\bundle k M}\langle\grad^h\varphi(f),\grad^{v,j}G^i\varphi(f)\rangle-\langle\grad^hG^i\varphi(f),\grad^{v,j}\varphi(f)\rangle\;d\mu+\\
& \mbox{\hspace{20mm}}+\int_{\bundle k M}\dvg^h\big(G^i\varphi(f)\grad^{v,j}\varphi(f)\big)\;d\mu\\
&=\int_{\bundle k M}\langle\grad^h\varphi(f),\grad^{v,j}G^i\varphi(f)\rangle-\langle\grad^hG^i\varphi(f),\grad^{v,j}\varphi(f)\rangle\;d\mu.
\end{align*}
Substituting in (\ref{+++}) we have the theorem.
\end{proof}

\section{Restriction to functions on the frame bundle $\bundle k M$}

In  this section we restrict to smooth functions on $\bundle k M$, re-formulating the integrated version of the Lifted Pestov Identity for this class of functions. Moreover, we derive an identity for smooth functions invariant under one of the frame flows.

\subsection{General formula}

%-----------------------------------------------------------------------------------------
\begin{thm}\label{thm:integral_pestov_bundle}
Let $\varphi\in C^{\infty}(\bundle k M)$ and $f=(v_1,\ldots,v_k)\in\bundle k M$. Then 
\begin{multline}\label{eq:integral_pestov_bundle}
k\|\grad^h\varphi\|^2_{L^2}-\frac{k+1}{2}\sum_{i=1}^k\|G^i\varphi\|^2_{L^2}-\int_{\bundle k M}\sum_{i,j=1}^k\langle R(\grad_{\liealgebra k}^{v,j}\varphi(f),v_j)v_i,\grad_{\liealgebra k}^{v,i}\varphi(f)\rangle\;d\mu=\\
=\sum_{i=1}^k\int_{\bundle k M}\langle\grad^h\varphi(f),\grad_{\liealgebra k}^{v,i}G^i\varphi(f)\rangle+\langle\grad^hG^i\varphi(f),\grad_{\liealgebra k}^{v,i}\varphi(f)\rangle\;d\mu
\end{multline}
where $L^2$ stands for the $L^2$-space on $\bundle k M$.
\end{thm}
%-------------------------------------------------------------------------------------------

\begin{proof}
Let $\tilde{\varphi}$ be a smooth extension of $\varphi$ on $\tkm k M$. We consider equation (\ref{eq:integral_pestov}) and we set $j=i$. Summing over $i=1,\ldots,k$ we obtain 
\begin{multline}\label{-}
k\|\grad^h\tilde{\varphi}\|^2_{L^2}-\int_{\bundle k M}\sum_{i,l=1}^k\langle R(\grad^{v,l}\tilde{\varphi}(f),v_l)v_i,\grad^{v,i}\tilde{\varphi}(f)\rangle\;d\mu=\\
=\int_{\bundle k M}\sum_{i=1}^k\langle\grad^h\tilde{\varphi}(f),\grad^{v,i}G^i\tilde{\varphi}(f)\rangle+\langle\grad^hG^i\tilde{\varphi}(f),\grad^{v,i}\tilde{\varphi}(f)\rangle\;d\mu.
\end{multline}
We consider the RHS of the above equation and using equations (\ref{eq:gradv_proj}) and (\ref{eq:flow_gradh}) we obtain
\begin{multline}\label{*}
\int_{\bundle k M}\langle\grad^h\tilde{\varphi}(f), \grad^{v,i}G^i\tilde{\varphi}(f)\rangle\;d\mu=\int_{\bundle k M}\langle\grad^h\tilde{\varphi}(f),\grad_{\liealgebra k }^{v,i}G^i\tilde{\varphi}(f)\rangle\;d\mu\\
+\frac{1}{2}\sum_{j=1}^k\int_{\bundle k M}G^j\tilde{\varphi}(f)\langle\grad^{v,i}G^i\tilde{\varphi}(f),v_j\rangle+G^j\tilde{\varphi}(f)\langle\grad^{v,j}G^i\tilde{\varphi}(f),v_i\rangle\;d\mu,
\end{multline}
and
\begin{multline}\label{**}
\int_{\bundle k M}\langle\grad^{v,i}\tilde{\varphi}(f),\grad^h G^i\tilde{\varphi}(f)\rangle\;d\mu=\int_{\bundle k M}\langle\grad^{v,i}_{\liealgebra k}\tilde{\varphi}(f),\grad^h G^i\tilde{\varphi}(f)\rangle\;d\mu+\\
+\frac{1}{2}\sum_{j=1}^k\int_{\bundle k M}G^jG^i\tilde{\varphi}(f)\langle\grad^{v,j}\tilde{\varphi}(f),v_i\rangle+G^jG^i\langle\grad^{v,i}\tilde{\varphi}(f),v_j\rangle\;d\mu\\
=\int_{\bundle k M}\langle\grad^{v,i}_{\liealgebra k}\tilde{\varphi}(f),\grad^h G^i\tilde{\varphi}(f)\rangle\;d\mu+\frac{(k+1)}{2}\|G^i\tilde{\varphi}(f)\|^2_{L^2}\\
-\frac{1}{2}\sum_{j=1}^k\int_{\bundle k M}G^i\tilde{\varphi}(f)\langle\grad^{v,i}G^j\tilde{\varphi}(f),v_j\rangle+G^i\tilde{\varphi}(f)\langle\grad^{v,j}G^j\tilde{\varphi}(f),v_i\rangle\;d\mu,
\end{multline}
where we used Lemma \ref{lem:integral_flow} and \ref{lem:G_gradv} for the second equality.

Adding (\ref{*}) and (\ref{**}) and taking the sum over $i$, the RHS of (\ref{-}) becomes 
\begin{multline}\label{--}
\int_{\bundle k M}\sum_{i=1}^k\langle\grad^h\tilde{\varphi}(f),\grad_{\liealgebra k}^{v,i}G^i\tilde{\varphi}(f)\rangle+\langle\grad^h G^i\tilde{\varphi}(f),\grad_{\liealgebra k}^{v,i}\tilde{\varphi}(f)\rangle\;d\mu\\
+\frac{k+1}{2}\sum_{i=1}^k\|G^i\tilde{\varphi}\|^2_{L^2}.
\end{multline} 
Finally, using (\ref{eq:gradv_proj}) and the antisymmetry of $R$, we have
\begin{equation}\label{eq:tensor1}
\sum_{i,j=1}^k\langle R(\grad^{v,j}\varphi(f),v_j)v_i,\grad^{v,i}\varphi(f)\rangle=\sum_{i,j=1}^k\langle R(\grad^{v,j}_{\liealgebra k}\varphi(f),v_j)v_i,\grad^{v,i}_{\liealgebra k}\varphi(f)\rangle.
\end{equation}
Substituting (\ref{--}) and (\ref{eq:tensor1}) in (\ref{-}) proves the theorem.
\end{proof}

\subsection{Consequence on invariant functions}

To conclude this section we present a new identity derived from (\ref{eq:integral_pestov_bundle}) assuming the function $\varphi$ to be invariant under one of the frame flows and choosing $k=n$. This identity generalises for $C^{\infty}$-functions on $\bundle k M$ for any $1\leq k< n$, with additional terms occurring and its proof involves a rather heavy calculation, which is not used in this article.
%---------------------------------------------------------------------------------------
\begin{cor}\label{cor:invariance}
Let $\varphi\in C^{\infty}(\bundle n M)$ and assume that it is invariant under the $i$-th frame flow, i.e., $G^i\varphi(f)=0$ for all $f\in\bundle n M$, and let $f=(v_1,\ldots,v_n)\in\bundle n M$. Then
\begin{equation}\label{eq:invariance}
\frac{1}{2}\sum_{j=1,j\neq i}^n\|G^j\varphi\|^2_{L^2}
=\sum_{j=1}^n\int_{\bundle n M}\langle R(w_j,v_j)w_i,v_i\rangle\;d\mu,
\end{equation}
where $w_i=\grad^{v,i}_{\liealgebra n}\varphi(f)$ for all $i=1,\ldots,k$ and $L^2$ stands for the $L^2$-space on $\bundle n M$.
\end{cor}
%-----------------------------------------------------------------------------------------

\begin{proof}
We prove the theorem for $i=1$. The other cases follow in the same way.

Let $f=(v_1,\ldots,v_n)\in\bundle n M$ and let $\tilde{\varphi}$ be a smooth extension of $\varphi$ on $\tkm n M$. 

We consider equation (\ref{eq:integral_pestov_bundle}) for $k=n$ and we aim to rewrite the horizontal gradient and its RHS in terms of $L^2$-norms of the generators of the frame flows and in terms of the Riemannian curvature tensor.

First of all, we observe that, in the case $k=n$, we have 
\begin{equation}\label{eq:1.0}
\|\grad^h\tvarphi\|^2_{L^2}=\sum_{i=1}^n\|G^i\tvarphi\|^2_{L^2}
\end{equation} 
as $v_1,\ldots,v_n$ is an orthonormal basis of $T_pM$ and $\grad^h\tvarphi(f)=\sum_{i=1}^nG^i\tvarphi(f) v_i$. 

We now look at the RHS of (\ref{eq:integral_pestov_bundle}).
From (\ref{eq:gradv_proj}) we derive 
\begin{equation}\label{eq:gradv_proj_inner}
\langle%\grad^{v,i}_{\liealgebra n}\varphi(f)
w_i,v_j\rangle=\frac{1}{2}\big(\langle\grad^{v,i}\tvarphi(f),v_j\rangle-
\langle\grad^{v,j}\tvarphi(f),v_i\rangle\big).
\end{equation}
Using this equation, Lemma \ref{lem:G_gradv} and the fact that $\grad^h\tvarphi(f)=\sum_{i=1}^nG^i\tvarphi(f) v_i$, we have
\begin{align}\label{eq:1.1}
\sum_{i=1}^n\int_{\bundle n M}\langle&\grad^h\tvarphi(f),\grad_{\liealgebra n}^{v,i}G^i\tvarphi(f)\rangle\;d\mu=\sum_{i=1}^n\int_{\bundle n M}\sum_{j=2,j\neq i}^nG^j\tvarphi(f)\langle\grad^{v,i}_{\liealgebra n}G^i\tvarphi(f),v_j\rangle\;d\mu\notag\\
&=\sum_{i,j=2,j\neq i}^n\int_{\bundle n M}\frac{1}{2}G^j\tvarphi(f)\Big[
\langle\grad^{v,i}G^i\tvarphi(f),v_j\rangle-\langle\grad^{v,j}
G^i\tvarphi(f),v_i\rangle\Big]\;d\mu\notag\\
&=\sum_{i,j=2,j\neq i}^n\int_{\bundle n M}\frac{1}{2}G^j\tvarphi(f)\Big[G^i\langle\grad^{v,i}\tvarphi(f),v_j\rangle-G^i\langle\grad^{v,j}\tvarphi(f),v_i\rangle+G^j\tvarphi(f)\Big]\;d\mu\notag\\
&\mbox{\hspace{10mm}}=\sum_{i,j=2,j\neq i}^n\int_{\bundle n M}-G^iG^j\tvarphi(f)\langle
%\grad^{v,i}_{\liealgebra n}\tvarphi(f)
w_i,v_j\rangle\;d\mu+\frac{n-2}{2}\sum_{i=2}^n\|G^i\tvarphi\|^2_{L^2},
\end{align}
and
\begin{equation}\label{eq:1.2}
\begin{split}
\sum_{i=2}^n\int_{\bundle n M}\langle\grad^hG^i\varphi(f),\grad_{\liealgebra n}^{v,i}\varphi(f)\rangle\;d\mu&=\sum_{i=2}^n\int_{\bundle n M}G^1G^i\tvarphi(f)\langle v_1,
%\grad^{v,i}_{\liealgebra n}\tvarphi(f)
w_i\rangle\;d\mu\\
&\mbox{\hspace{5mm}}+\sum_{i,j=2,j\neq i}^n\int_{\bundle n M}G^jG^i\tvarphi(f)\langle
%\grad^{v,i}_{\liealgebra k}\tvarphi(f)
w_i,v_j\rangle\;d\mu.
\end{split}
\end{equation}
Summing (\ref{eq:1.1}) and (\ref{eq:1.2}), using (\ref{eq:sym_flow}) and the skew-symmetry of the matrix $(\langle w_i,v_j\rangle)_{i,j}$, we have
\begin{align}\label{eq:1.3}
\sum_{i=1}^n&\int_{\bundle k M}\langle\grad^h\varphi(f),\grad_{\liealgebra n}^{v,i}G^i\varphi(f)\rangle+\langle\grad^hG^i\varphi(f),\grad_{\liealgebra n}^{v,i}\varphi(f)\rangle\;d\mu=\notag\\
&=\frac{n-2}{2}\sum_{i=2}^n\|G^i\tilde{\varphi}\|^2_{L^2}+\sum_{i=2}^n\int_{\bundle n M}G^1G^i\tvarphi(f)\langle v_1
%,\grad^{v,i}_{\liealgebra n}\tvarphi(f)
w_i\rangle\;d\mu+\notag\\
& \mbox{\hspace{27mm}}+\sum_{i,j=2,i\neq j}^n\int_{\bundle n M}(G^jG^i-G^iG^j)\tvarphi(f)\langle v_j,
%\grad^{v,i}_{\liealgebra n}\tvarphi(f)
w_i\rangle\;d\mu\notag\\
&=\frac{n-2}{2}\sum_{i=2}^n\|G^i\tilde{\varphi}\|^2_{L^2}-\int_{\bundle n M}\sum_{i=2}^n\sum_{l=1}^n\langle R(w_l,v_l)v_i,w_i\rangle\;d\mu.
\end{align}
Substituting (\ref{eq:1.0}) and (\ref{eq:1.3}) in (\ref{eq:integral_pestov_bundle}) concludes the proof.
\end{proof}

\section{From frame bundles to oriented Grassmannians}

In this section we apply the theory developed previously to functions defined on the oriented $k$-th Grassmannian of $M$. In what follows, we present a dictionary between the world of frame bundles and frame flows and the world of oriented Grassmannian and parallel transports. This allows us to prove an invariance property of $C^{\infty}$-functions on the oriented Grassmannian, which is presented at the end of the section.

%---------------------------------------------------------------------------------------
\subsection{Oriented Grassmannians and intrinsic parallel transports} 

The oriented $k$-th Grassmannian of $M$, $\grass k M$ for $1\leq k\leq n=\dim M$, is the union of all $k$-planes of $T_pM$ for all $p\in M$ together with an intrinsic orientation,i.e.,
\begin{equation*}
\grass k M=\bigunion_{p\in M}\big\{A_{or}\;\big\vert\; A\subset T_pM, \dim A=k\big\},
\end{equation*} 
where $A_{or}$ is a subspace of $T_pM$ with an additional choice of intrinsic orientation.
 
This is a $2$-fold covering of the non-oriented $k$-th Grassmannian of $M$ and, in the special case of $k=1$, $\grass 1 M=SM$, the unit tangent bundle of $M$.

There is a canonical projection 
\begin{equation}\label{eq:proj_frametograss}
\begin{array}{c}
\tilde{\pi}:\bundle n M\longrightarrow\grass k M,\\
\mbox{\vspace{5mm}}f=(v_1,\ldots,v_n)\mapsto\big(span\{v_1,\ldots,v_k\},(v_1,\ldots,v_k)\big),
\end{array}
\end{equation}
where $(v_1,\ldots,v_k)$ stands for intrinsic orientation.

Therefore, any function $\varphi\in C^{\infty}(\grass k M)$ can be extended to a function $\phi\in C^{\infty}(\bundle n M)$ by setting $\phi=\varphi\circ\tilde{\pi}$. The function $\phi$ has the following two important properties. 

Firstly, $\phi$ is invariant under the action of matrices of the form 
$\left(
\begin{array}{cc}
SO(k) & 0\\
0& SO(n-k)\\
\end{array}
\right)$
since $\varphi$ is invariant under the action of $SO(k)$.

Secondly, its vertical gradients satisfy the following.
%---------------------------------------------------------------------------------------
\begin{lemma}\label{lem:gradv_span}
Let $\varphi\in C^{\infty}(\grass k M)$ and $\phi=\varphi\circ\tilde{\pi}\in C^{\infty}(\bundle n M)$. Then, 
\begin{itemize}
\item [(i)] $\grad^{v,i}_{\liealgebra n}\phi(f)\in span\{v_1,\ldots,v_k\}$ if $i\geq k+1$;
\item [(ii)]  $\grad^{v,i}_{\liealgebra n}\phi(f)\in span\{v_{k+1},\ldots,v_n\}$ if $i=1,\ldots k$.
\end{itemize}
\end{lemma}
%-----------------------------------------------------------------------------------------
\begin{proof}
We prove  $(i)$ and $(ii)$ showing that $\langle\grad^{v,i}_{\liealgebra n}\phi(f),v_j\rangle=0$ for $i,j\geq k+1$ and $\langle\grad^{v,i}_{\liealgebra n}\phi(f),v_j\rangle=0$ for $i,j\leq k$.

Let $\tilde{\phi}$ be a smooth extension of $\phi$ on $\tkm n M$ constructed as follows.

We define the functions $h:\tkm n M\rightarrow [0,\infty]$ such that $h(w_1,\ldots,w_n)=\det(\langle w_i,w_j\rangle)_{ij}$ and $\psi:\{(w_1,\ldots,w_n)\mbox{ lin. indep.}\}\rightarrow\bundle n M$ to be the Gram-Schmidt process. 

Furthermore, let $H:[0,\infty]\rightarrow [0,1]$ be a cut off function such that $H(x)=0$ for $x\leq\frac{1}{2}$, $H(x)=1$ for $x\geq\frac{3}{4}$ and $0\leq H(x)\leq 1$ for $\frac{1}{2}\leq x\leq\frac{3}{4}$. Then,
$$\tilde{\phi}(w_1,\ldots,w_n)=\left\{
\begin{array}{lr}
H\big(h(w_1,\ldots,w_n)\big)\cdot(\phi\circ\psi)(w_1,\ldots,w_n) & \qquad\{w_1,\ldots,w_n\}\mbox{ lin. indep.}\\
\mbox{\vspace{3mm}}0 & \mbox{\vspace{3mm}otherwise}
\end{array}\right.$$
Note that $\tilde{\phi}(w_1,\ldots, w_l+tw_j,\ldots,w_n)=\tilde{\phi}(w_1,\ldots,w_n)$ for all $l\geq k+1$.

This implies that, for $i,j\geq k+1$, we have
\begin{align*}
\langle\grad&^{v,i}_{\liealgebra n}\phi(f),v_j\rangle=\frac{1}{2}\langle\grad^{v,i}\tilde{\phi}(f),v_j\rangle-\frac{1}{2}\langle\grad^{v,j}\tilde{\phi}(f),v_i\rangle\\
&=\frac{1}{2}\der t\Bigvert_{t=0}\Big(\tilde{\phi}(v_1,\ldots,v_k,\ldots,v_i+tv_j,\ldots,v_n)-\tilde{\phi}(v_1,\ldots,v_k,\ldots,v_j+tv_i,\ldots,v_n)\Big)=0,
\end{align*}
which proves $(i)$. 

Part $(ii)$ follows from the fact that $\tilde{\phi}$ depends only on the plane spanned by the first $k$ vectors.
\end{proof}

Next, we define the notion of \textit{intrinsic} parallel transport of oriented $k$-planes of $\grass k M$.

\begin{defin}
Let $(A_{or})_v(t)$ denote the parallel transport of the oriented $k$-plane $A_{or}\in\grass k M$ along a curve $c_v$ on $M$ with $c'_v(0)=v$.  We say that the parallel transport is \textit{intrinsic} if the vector $v$ belongs to $A_{or}$ and we call it \textit{non-intrinsic} otherwise.
\end{defin}

This definition gives the following link between frame flows on $\phi=\varphi\circ\tilde{\pi}$ and intrinsic parallel transports applied to $\varphi$. 

Let $f=(v_1,\ldots, v_n)\in\bundle k M$ and $A_{or}\in\grass k M$ such that $A_{or}=\tilde{\pi}(f)$, then 
%---------------------------------------------------------------------------------------
\begin{equation}\label{eq:flow_transport}
G^i\phi(f)=G^i(\varphi\circ\tilde{\pi})(f)
=\der t\Bigvert_{t=0}(\varphi\circ\tilde{\pi})(f_{v_i}(t))
=\der t\Bigvert_{t=0}\varphi((A_{or})_{v_i}(t)).
\end{equation}
%--------------------------------------------------------------------------------------
In particular, this implies that $\varphi$ is invariant under intrinsic parallel transport if and only if $\phi$ is invariant under the $i$-th frame flow, for $i=1,\ldots k$.

%----------------------------------------------------------------------------------------
\subsection{Invariance property}

In what follows we assume that $M$ has non-positive curvature operator $\mathcal{R}$. The curvature operator is a linear operator $\mathcal{R}:\Lambda^2(TM)\longrightarrow\Lambda^2(TM)$ defined as
\begin{equation}
\langle\mathcal{R}(X\wedge Y),Z\wedge W\rangle_{\Lambda^2(TM)}=\langle R(X,Y)W,Z\rangle_{TM}
\end{equation}
for all vector fields $X,Y,Z,W$ on $M$.

The curvature operator $\mathcal{R}$ is symmetric and we say that $\mathcal{R}$ is non-positive ($\mathcal{R}\leq 0$) if all of its real eigenvalues are non-positive. 

A manifold $M$ with non-positive curvature operator has non-positive curvature but the inverse implication is not true. For details on this topic we refer the reader to, e.g., \cite{aravinda:10} and \cite{aravinda-farell:04}.

We now prove our main result on $C^{\infty}$-functions on $\grass k M$, i.e. Theorem 1.1.
%---------------------------------------------------------------------------------------
\bigskip

\noindent\textbf{Theorem \ref{thm:invariance_grass_intro}.}
\emph{Let $M$ be a compact $n$-dimensional Riemannian manifold with $\mathcal{R}\leq 0$, $1\leq k\leq n$, and let $\varphi\in C^{\infty}(\grass k M)$. If $\varphi$ is invariant under all intrinsic parallel transports then it is also invariant under all parallel transports.}

\bigskip
%-----------------------------------------------------------------------------------

Before proceeding with the proof we first consider the case $k=1$.

As we have already observed, in this case $\grass 1 M=SM$. Moreover, the intrinsic parallel transport corresponds to the geodesic flow and it follows from the proof of Thereom \ref{thm:invariance_grass_intro} that the assumption on the non-positivity of the curvature operator can be weakened to the non-positivity of the sectional curvature. In this case, we recover the following unpublished result of G. Knieper \cite{knieper:preprint}.

%----------------------------------------------------------------------------------------
\begin{cor}\label{cor:knieper_result}
Let $M$ be a compact Riemannian manifold with non-positive curvature. Let $\varphi\in C^{\infty}(SM)$ invariant under the geodesic flow, then $\varphi$ is also invariant under parallel transport.
\end{cor}
%-------------------------------------------------------------------------------------

\begin{proof}[Proof of Thm. \ref{thm:invariance_grass_intro}]
Let $\phi=\varphi\circ\tilde{\pi}\in C^{\infty}(\bundle n M)$.

Since $\varphi$ is invariant under intrinsic parallel transports, $\phi$ is invariant under $G^1,\ldots, G^k$ due to (\ref{eq:flow_transport}).

Considering equation (\ref{eq:invariance}) and summing over $i=1,\ldots,k$ we obtain
\begin{equation}\label{@}
\begin{split}
\frac{k}{2}\sum_{j=1}^n\|G^j\phi\|^2_{L^2(\bundle n M)}& =\int_{\bundle n M}\sum_{i=1}^k\sum_{j=1}^n\langle R(w_j,v_j)v_i,w_i\rangle\;d\mu\\
&=\int_{\bundle n M}\langle\mathcal{R}\Big(\sum_{j=1}^nw_j\wedge v_j\Big),\sum_{i=1}^kw_i\wedge v_i\rangle_{\Lambda^2(TM)}\;d\mu
\end{split}
\end{equation}
where $w_i=\grad^{v,i}_{\liealgebra n}\phi(f)$.

Since the matrix $\big(\langle w_i,v_j\rangle\big)_{i,j}$ is skew-symmetric and making use of Lemma \ref{lem:gradv_span}, we obtain
\begin{align*}
\sum_{j=1}^nw_j\wedge v_j&=\sum_{j=1}^n\sum_{l=1,l\neq j}^n\langle w_j,v_l\rangle v_l\wedge v_i=\sum_{j=1}^k\sum_{l=k+1}^n\langle w_j,v_l\rangle v_l\wedge v_j+\sum_{j=k+1}^n\sum_{l=1}^k\langle w_j,v_l\rangle v_l\wedge v_j\\
&=2\sum_{j=1}^kw_j\wedge v_j.
\end{align*}
Since $\mathcal{R}\leq 0$, the RHS of (\ref{@}) is non-positive forcing the LHS to be zero. 

We conclude that $\phi$ is invariant under all frame flows, and so $\varphi$ is invariant under all parallel transports.  
\end{proof}

In view of Theorem \ref{thm:invariance_grass_intro}, it is natural to investigate density properties of orbits of $k$-planes $A$ under all intrinsic parallel transports. In fact, let $I(A)$ be the set of all $k$-planes obtained by finitely many moves along intrinsic parallel transport and $G(A)$ be the set of all $k$-planes obtained by finitely many moves along general parallel transport. Theorem \ref{thm:invariance_grass_intro} suggests that even though there might be many $k$-planes $A$ such that $I(A)$ is much smaller and not dense in $G(A)$, there might always be a $k$-planes $A'$ arbitrarily close to $A$ such that $I(A')$ is dense in $G(A')$. This is at least true in the case of the flat torus and in the case of constant negative curvature. In the general case, this seems to be a difficult question to answer.

However, we can give an answer to the related, easier question, whether a smooth function $\varphi$ invariant under all intrinsic parallel transports is necessarily constant. 

%------------------------------------------------------------------------------------------
\bigskip

\noindent\textbf{Proposition \ref{prop:constant_functions_intro}}
\emph{Let $M$ be a non-flat, compact Riemannian manifold with non-positive curvature operator $\mathcal{R}$. Then the following holds:
\begin{itemize}
\item [(i)] If $M$ is either a K\"ahler manifold of real dimension $2n\geq 4$ or a Quaternion-K\"ahler manifold of real dimension $4n\geq 8$ or a locally symmetric space of non-constant curvature (i.e., not the real hyperbolic space), then there exist smooth,\emph{ non-constant} functions on $\grass 2 M$ or $\grass 4 M$ which are invariant under all parallel transports.
\item [(ii)] If $M$ is not one of the exceptions in $(i)$, then, for all $k\leq\dim M$, any smooth function on $\grass k M$ invariant under all parallel transports is necessarily constant.
\end{itemize}}

%-----------------------------------------------------------------------------------------
\begin{proof}
$(i)$ First, let $M$ be a K\"ahler manifold of real dimension $2n\geq 4$. The almost complex structure $J$ is parallel and it gives rise to a smooth function $\varphi$ on oriented $2$-planes, which is invariant under all parallel transports but it is not constant. This function is defined via $\varphi(A_{or})=\langle v_1,Jv_2\rangle$ where $v_1,v_2$ is an oriented orthonormal basis of $A_{or}\in\grass 2 M$.

Secondly, let $M$ be a Quaternion-K\"ahler manifold of real dimension $4n\geq 8$ with non-positive curvature operator. The canonical $4$-forms $\Omega$ globally defined on $M$ is parallel (see, e.g., \cite{gray:69} or \cite{ishihara:74}). This gives rise to the smooth function  $\varphi:\grass 4 M\rightarrow\R$, defined as $\varphi(A_{or})=\Omega_p(v_1,\ldots,v_4)$ where $v_1,\ldots,v_4$ is an oriented orthonormal basis of $A_{or}\in\grass 4 M$ and $A_{or}\in T_pM$. This function is invariant under all parallel transports and non constant.
 
Finally, let $M$ be a locally symmetric space of non-constant non-positive curvature, then $M$ is a compact quotient of a symmetric space with non-constant curvature and its Riemannian curvature tensor is parallel. We consider the function $\varphi\in C^{\infty}(\grass 2 M)$ such that $\varphi(A_{or})=\langle R(v_1,v_2)v_2,v_1\rangle$ where $v_1,v_2$ is an oriented orthonormal basis of $A_{or}\in\grass 2 M$. Now, $\varphi$ is invariant under all parallel transports but it is not constant.

$(ii)$ If $M$ is a $n$-dimensional manifold which is not one of the exceptions above, the holonomy of $M$ is $SO(n)$ (see \cite{berger:03} or \cite{besse:87}). Therefore, any smooth function on $\grass k M$ invariant under all intrinsic parallel transports is also invariant under the non-intrinsic parallel transports by Theorem \ref{thm:invariance_grass_intro} and, hence, is constant due to the transitive action of $SO(n)$ on oriented $k$-planes in the tangent space.                                                    
\end{proof}

\begin{rmk}
It seems to be an open question whether there exist compact non-locally symmetric Quaternion-K\"ahler manifolds with non-positive curvature operator. We are grateful to Vicente Cort\'es for references in connection to this question (\cite{CDL:14}, \cite{CNS:13}, \cite{lebrun:91}, \cite{lebrun:88}).
\end{rmk}

%----------------------------BIBLIOGRAPHY---------------------------------


\begin{thebibliography}{100}

\bibitem {ainsworth:13} G.~Ainsworth, \emph{The attenuated magnetic ray transform on surfaces}, Inverse. Probl. Imaging, \textbf{7} (2013), 27--46.

\bibitem {anik-rom:97} Y.~Anikonov, V.~Romanov, \emph{On uniqueness of determination of a form of first degree by its integrals along geodesics}, J. Inverse Ill-Posed Probl, \textbf{5} (1997), 467--480. 

\bibitem {aravinda:10} C.~S.~Aravinda, \emph{Curvature vs. Curvature Operator}, Special issue of the RMS Newsletter commemorating ICM 2010 in India, \textbf{19} (2010), 109--118.

\bibitem {aravinda-farell:04} C.~S.~Aravinda, T.~Farell, \emph{Nonpositivity: curvature vs curvature operator}, Proceedings of the American Mathematical Society, \textbf{133} (2005), 191--192.

\bibitem {ballmann:} W.~Ballmann, \emph{Lectures on K\"ahler manifolds}, European Mathematical Society, 2006.

\bibitem {berger:03} M.~Berger, \emph{A panoramic view of Riemannian Geometry}, Springer-Verlag, Berlin, 2003.

\bibitem {besse:87} A.~Besse, \emph{Einstein Manifolds}, Springer-Verlag, Berlin, 1987.

\bibitem {brin-gromov:80} M.~Brin, M.~Gromov, \emph{On the ergodicity of frame flows}, Inventiones Math., \textbf{60} (1980), 1--7.

\bibitem {brin-karcher:84} M.~Brin, H.~Karcher, \emph{Frame flows on manifolds with pinched negative curvature}, Compositio Math., \textbf{52} (3) (1984), 275--297.

\bibitem {burns-pollicott:03} K.~Burns, M.~Pollicott, \emph{Stable ergodicity and frame flows}, Geom. Dedicata, \textbf{98} (2003), 189--210.

\bibitem {CDL:14} V.~Cort\'es, M.~Dyckmanns, D.~Lindemann,\emph{Classification of complete projective special real surfaces} Proc. London Math. Soc., \textbf{109} (3) (2014), 423–-445.

\bibitem {CNS:13} V.~Cort\'es, M.~Nardmann, S.~Suhr, \emph{Completeness of hyperbolic centroaffine hypersurfaces}, arXiv:1305.3549.

\bibitem {croke-shar:98} C.~Croke, V.~A.~Sharafutdinov, \emph{Spectral rigidity of compact negatively curved manifolds}, Topology, \textbf{37} (1998), 1265--1273.

\bibitem {dar-pat:05} N.~S.~Dairbekov, G.~P.~Paternain, \emph{Longitudinal KAM-cocycles and action spectra of magnetic flows}, Math. Res. Lett., \textbf{12} (2005), 719--729.

\bibitem {dar-pat:08} N.~S.~Dairbekov, G.~P.~Paternain, \emph{On the cohomological equation of magnetic flows}, Matem\'atica Contempor\^anea, \textbf{34} (2008), 155--193.

\bibitem {dar-pat:082} N.~S.~Dairbekov, G.~P.~Paternain, \emph{Rigidity properties of Anosov optical hypersurfaces}, Ergodic Theory and Dynam. Systems, \textbf{28} (2008), 707--737.

\bibitem {dar-pat-etal:07} N.~S.~Dairbekov, G.~P.~Paternain, P.~Stefanov, G.~Uhlmann, \emph{The boundary rigidity problem in the presence of a magnetic field}, Adv. Math., \textbf{21} (2007), 535--609.

\bibitem {feres:98} R.~Feres, \emph{Dynamical systems and semisimple groups. An introduction.}, Cambridge University Press, Cambridge, 1998. 

\bibitem {gray:69} A.~Grey, \emph{A note on manifolds whose holonomy group is a subgroup of $Sp(n)\cdot Sp(1)$}, Michigan Math. J., \textbf{16} (2) (1969), 125--128.

\bibitem {gui-kaz:80} V.~Guillemin, D.~Kazhadan, \emph{Some inverse spectral results for negatively curved 2-manifolds}, Topology, \textbf{19} (1980), 301--312.

\bibitem {ishihara:74} S.~Ishihara, \emph{Quaternian K\"ahlerian manifolds}, J. Differential Geometry, \textbf{9} (1974), 483--500.

\bibitem {knieper:02} G.~Knieper, \emph{Hyperbolic dynamics and Riemannian geometry}, in  Handbook of dynamical systems, Vol. 1A, 453--545, North-Holland, Amsterdam, 2002.

\bibitem {knieper:preprint}  G.~Knieper, \emph{A commutator formula for the geodesic flow on a compact Riemannian manifold}, pre-print.

\bibitem {KNI} S.~Kobayashi, K.~Nomizu, \emph{Foundations of Differential Geometry}, Vol I, Wiley Interscience Publishers, John Wiley and Sons Inc., New York, 1996.

\bibitem {lebrun:91} C.~LeBrun, \emph{On complete Quaternionic-K\"ahler manifolds}, Duke Math. J., \textbf{63} (3) (1991), 723--743.

\bibitem{lebrun:88} C.~LeBrun, \emph{A rigidity theorem for Quaternionic-K\"ahler manifolds}, Proc. Am. Math. Soc, \textbf{103} (4) (1988), 1205--1208.

\bibitem {maus:05} N.~Maus, \emph{Der geod\"atische Fluss auf symmetrischen R\"aumen}, Diplomarbeit Bochum, 2005.

\bibitem {min-oo:85} M.~Min-Oo, \emph{Spectral rigidity for manifolds with negative curvature operator}, in Nonlinear problems in geometry, 99--103, Contemp. Math., \textbf{51}, Amer. Math. Soc., Providence, RI, 1986.

\bibitem {pat-salo-uhl:13} G.~P.~Paternain, M.~Salo, G.~Uhlmann, \emph{Tensor tomography on surfaces}, Invent. Math., \textbf{193} (2013), 229--247.

\bibitem {pat-salo-uhl:14} G.~P.~Paternain, M.~Salo, G.~Uhlmann, \emph{Tensor tomography: progress and challenges}, Chin. Ann. Math., \textbf{35B} (3) (2014), 399--428.

\bibitem {pes-shar:88} L.~N.~Pestov, V.~A.~Sharafutdinov, \emph{Integral Geometry of tensor fields on a manifold of negative curvature}, (Russian) Sibirsk. Mat. Zh. \textbf{29} (3) (1988) 114--130, 221; translation in Siberian Math. J. \textbf{29} (3) (1988), 427--441.

\bibitem {shar:book} V.~A.~Sharafutdinov, \emph{Integral Geometry of Tensor Fields}, VSP, Utrecht, the Netherlands, 1994.

\end{thebibliography}
\end{document}